\documentclass[12pt]{article}
\usepackage[top=1in, bottom=1in, left=1in, right=1in]{geometry}
\usepackage{amsmath}
\usepackage{amsfonts}
\usepackage{amsthm}
\usepackage{amssymb, setspace}
\usepackage{color,graphicx, epsfig, geometry, hyperref,fancyhdr}
\newtheorem{theorem}{Theorem}[section]
\usepackage{framed}
\usepackage{float} 
\newtheorem{definition}{Definition}[section] 
\newtheorem{remark}{Remark}[section] 
\newtheorem{corollary}{Corollary}[section] 
\newtheorem{lemma}{Lemma}[section]

\newcommand{\R}{\mathbb{R}}
\newcommand{\C}{\mathbb{C}}

\newcommand{\grad}{\nabla}

\newcommand{\weakc}{\rightharpoonup}

\begin{document}
\setlength{\parskip}{1mm}
\setlength{\oddsidemargin}{0.1in}
\setlength{\evensidemargin}{0.1in}
\lhead{}
\rhead{}

\begin{center}
{\bf THE FACTORIZATION METHOD FOR A DEFECTIVE REGION IN AN ANISOTROPIC MEDIA}
\end{center}
\vspace{0.05in}
\begin{center}
Fioralba Cakoni\\
Department of Mathematical Sciences\\
University of Delaware Newark\\
Delaware 19716-2553, USA \\ 
E-mail address: cakoni@math.udel.edu\\
\vspace{0.3in}

Isaac Harris \\
Department of Mathematical Sciences\\
University of Delaware Newark\\
Delaware 19716-2553, USA \\
E-mail address: iharris@udel.edu
\end{center}

\begin{abstract}
In this paper we consider the inverse acoustic scattering (in $\R^3$) or electromagnetic scattering (in $\R^2$, for the scalar TE-polarization case) problem of reconstructing possibly multiple defective penetrable regions in a known anisotropic material of compact support. We develop the factorization method for a non-absorbing anisotropic background media containing penetrable defects. In particular, under appropriate assumptions on the anisotropic material  properties of the media we develop a rigorous characterization for the support of the  defective regions from the given far field measurements. Finally we present some numerical examples  in the two dimensional case to demonstrate the feasibility of our reconstruction method including examples for the case when the defects are voids (i.e. subregions with refractive index the same as the background outside the inhomogeneous hosting media).
\end{abstract}

{\bf Keywords}: factorization method, anisotropic materials, non-destructive testing, inverse scattering problem.

\section{Introduction}
Nondestructive testing of  exotic materials using acoustic or electromagnetic waves is an important engineering problem. The inverse problem that we are interested in is to determine the shape and position of defects in a known anisotropic material of compact support. This problem arises for example in nondestructive testing of airplane canopies. Using Newton type optimization techniques it is possible to  reconstruct the refractive index of the defect (see e.g. \cite{coltonkress}, \cite{hohage} and the references therein for inverse medium problem in a homogeneous background). However,  such methods require good a priori information about  the type and the number of components of possible defects, and they are problematic for anisotropic media due to lack of uniqueness. Alternative methods for solving  the inhomogeneous media inverse problem that come under the general title of qualitative methods, such as sampling methods,  practically do not require any a priori information but   as oppose to nonlinear optimization techniques only seek limited information about the defects.  It has been shown in \cite{mypaper1} that, when the defect is a void(s) (i.e. subregions with refractive index the same as the background outside the inhomogeneous hosting media) one can qualitatively obtain information about the size of the void(s) from far field data using the  corresponding transmission eigenvalues (see Definition \ref{trig} in this paper). In this paper we develop a factorization method (see \cite{armin}, \cite{kirschbook} and the references therein),  to reconstruct the support of the defective region. A similar problem was considered in \cite{fmconstant} where it is assumed that the background media is piecewise homogeneous with a sound-soft obstacle embedded in it. Also in \cite{fmiso} the factorization method was developed for non-absorbing inhomogeneous media embedded in a piecewise homogeneous background. We remark that other qualitative methods such as the linear sampling method and reciprocity gap functional have been developed for inhomogeneous (possibly anisotropic) background \cite{fares}, \cite{haddar}, \cite{coyle}.  We remark that the factorization method is the most rigorously justified  technique within the class of qualitative methods in inverse scattering. 

Motivated by nondestructive testing of anisotropic material, we develop the factorization method for  determining the support of a penetrable (possibly anisotropic) defective region embedded in a known anisotropic media of compact support  sitting in a homogeneous background.  The factorization method gives a rigorous characterization of the support of the defect in terms of the far field operator provided that the background is known hence providing also a uniqueness result. Note that for anisotropic defects the unique determination of the support is the best we can hope, since in general it is well known that the  matrix-valued refractive index is not uniquely determined. We note that, the factorization method for this configuration involves the computation  of the far field pattern of Green's function for the inhomogeneous background media. However for the case of anisotropic homogeneous media we extend the result in \cite{fmconstant} and provide a simple formula to compute the far field pattern of the background Green's function in terms of the total field due to the background. As a particular application of this study, we consider the determination of the  support of voids inside  a known anisotropic media.

The paper is structured as follows. After formulating the scattering problem in the next section,  we construct a factorization of  the far field field operator which is defined in terms of the measured far field data and the far field pattern of the scattered field due to the background. Then in Section 4 we use the main factorization theorems in  \cite{kirschbook}  and  \cite{armin} to  derive an indicator function for  the support of the defect $D_0$ embedded in a known anisotropic media with support $D$ (see Figure  \ref{geoexample1}) under reasonable assumptions on the constitutive parameters of the background and the defect. In the last section we present some numerical examples to show the viability of our reconstruction method. We remark that for standard asymptotic expressions in scattering theory used here, we refer the reader to  \cite{p1} for the case of  ${\mathbb R}^2$ and  to \cite{coltonkress} for the case of  ${\mathbb R}^3$.

\section{Formulation of the Problem}\label{problem}
We start by introducing  the scattering problem for a ``healthy" and  ``faulty"  material in ${\mathbb R}^m$, $m=2,3$. 
 
 \begin{figure}[H]
\centering
\includegraphics[scale=0.25]{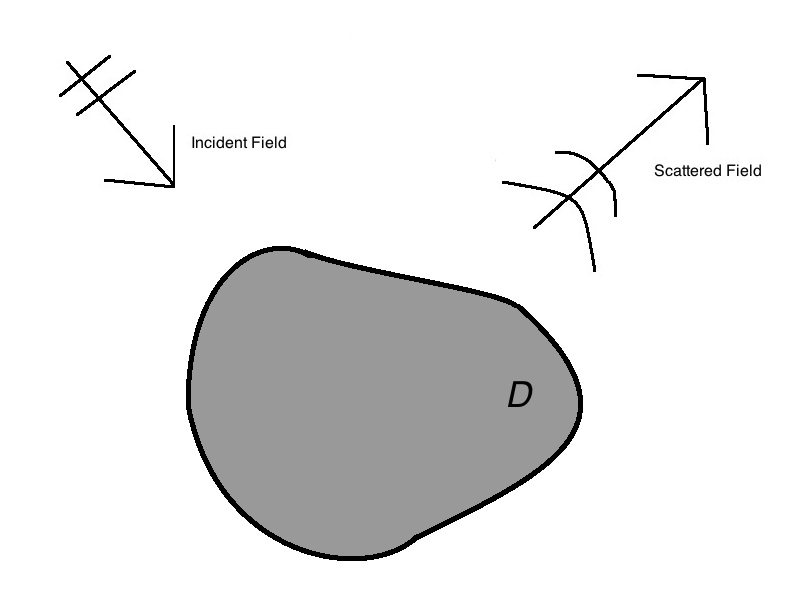}\includegraphics[scale=0.25]{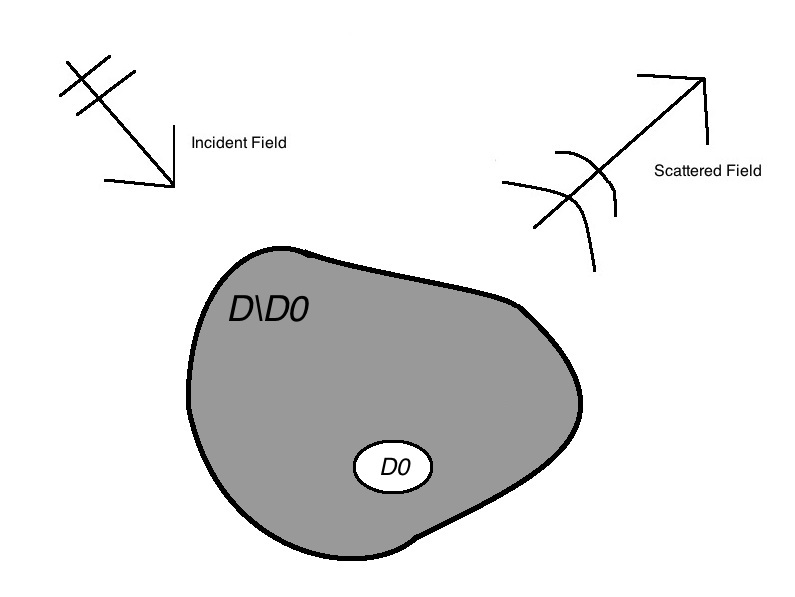}\\
\caption{Example geometry of the scattering of a medium without and with a defective region. }
\label{geoexample1}
\end{figure}

 To this end, let $D \subset \R^m$ be a bounded simply connected open set having piece-wise smooth boundary $\partial D$ with $\nu$ being the unit outward normal to the boundary. We assume that the constitutive parameters of the media in $D$ are represented by a real-valued  symmetric matrix $\tilde A\in  C^{1} \left( D, \R^{m \times m} \right)$ and a real valued function $\tilde n \in C^1(D)$  such that  $\overline{\xi}\cdot \tilde A(x) \xi\geq a_{min} |\xi|^2>0$ and $\tilde n(x)\geq n_{min}>0$ for almost all $x\in D$ and all $\xi\in {\mathbb C}^m$. Outside $D$ the background media is homogeneous isotropic with refractive index scaled to one. We denote by $A$ and $n$ the constitutive parameters of the anisotropic background ${\mathbb R}^m$ given by
 $$A(x):=\left\{\begin{array}{rrcll} \tilde A(x) & x\in D\;\;\; \\
 I \;\;\; & \;\;\; \;\; x\in{\mathbb R}^m\setminus \overline{D}
 \end{array}\right.  \qquad \qquad  n(x):=\left\{\begin{array}{rrcll} \tilde n(x) & \;\;\; x\in D\;\;\; \\
 1 \;\;\; & \;\;\;\; x\in{\mathbb R}^m\setminus \overline{D}
 \end{array}\right.$$
 where $I$ is the identity matrix. Note that the support of $A-I$ and $n-1$ is $\overline{D}$. Now the scattering of an incident plane wave $e^{ikx \cdot d}$, where $d$ is a unitary vector, by the ``healthy" anisotropic material  (i.e. without defects) is mathematically formulated as: find $u_b \in H^1_{loc}(\R^m)$ with $u_b =u^s_b + e^{ikx \cdot d}$  such that
\begin{eqnarray}
\grad \cdot A(x) \grad u_b +k^2 n(x) u_b=0 \,  &\textrm{ in }& \,  \R^m  \label{bgp1}\\
\lim\limits_{r \rightarrow \infty} r^{\frac{m-1}{2}} \left( \frac{\partial u^s_b}{\partial r} -iku^s_b \right)=0 \label{bgp3}
\end{eqnarray}
where the radiation condition (\ref{bgp3}) is satisfied uniformly with respect to $\hat x=x/|x|$.  We recall that  (\ref{bgp1}) implies that across the interface $\partial D$ we have
$$ \frac{\partial u_b^-}{\partial \nu_{A}} =\frac{\partial u_b^+}{\partial \nu}\qquad \mbox{on}\;\; \partial D$$
where  the superscripts $+$ and $-$ for a generic function indicates the trace on the boundary taken from the exterior or interior of its surrounding domain, respectively. 
Here $u_b$ is the total field  in  the background (including the homogeneous part and the anisotropic media of compact support $D$)  and $u^s_b$ is the scattered field  due to the anisotropic region $D$ of the background. It is known  that the scattered field $u_b^s(\cdot,d)$ which depends on the incident direction $d$, has the following asymptotic expansion 
$$u^s_b(x,d)=\frac{e^{ik|x|}}{|x|^{\frac{m-1}{2}}} \left\{u_b^{\infty}(\hat{x}, d ) + \mathcal{O} \left( \frac{1}{|x|}\right) \right\}\; \textrm{  as  } \;  |x| \to \infty $$
where $\hat x:=x/|x|$ and $u_b^{\infty}(\hat x, d) $, which depends on the incident direction  $d$ and observation direction $\hat{x}$, is the corresponding far field pattern.  The far field pattern is given by the integral representation
\begin{eqnarray}
u_b^\infty(\hat{x}, d )= \gamma_m \int\limits_{\partial \Omega} \left(u_b^s(y,d) \frac{\partial e^{-ik \hat{x}\cdot y}}{\partial \nu_y}  - \frac{\partial u_b^s(y,d) }{\partial \nu_y}e^{-ik \hat{x}\cdot y}\right) \, ds_y \label{ffpdef}
\end{eqnarray}
where the constant $\gamma_m$, $m=2,3$ is given by $\gamma_2= \frac{e^{i \pi/4}}{\sqrt{8 \pi k}} $ and $\gamma_3= \frac{1}{4 \pi }$ and the region $\Omega$ is any subset of $\R^m$ such that $D \subseteq \Omega$.  We now define the far field operator for the background scattering problem as $F_b: L^2(\mathbb{S}) \longmapsto  L^2(\mathbb{S})$  
$$(F_b g)(\hat{x}):=\int\limits_{\mathbb{S}} u_b^{\infty}(\hat{x}, d ) g(d) \, ds(d), \, \, \;\, \,  g \in L^2(\mathbb{S}) $$
where $ \mathbb{S}=\{ {x} \in \R^m\,  :  \, |{x}|=1 \}$ is the unit circle or sphere. For later use we introduce the scattering operator associated with this scattering problem, which plays an essential role in our factorization in the follow section. 
\begin{definition} \label{deft} The  scattering operator ${\mathcal S}_b:L^2(\mathbb{S}) \to L^2(\mathbb{S})$  for \eqref{bgp1}-\eqref{bgp3} is defined by 
\begin{eqnarray}
\mathcal{S}_b=I+2ik\gamma_m F_b. \label{sop}
\end{eqnarray}
Since $A$ and $n$ are real valued, the scattering operator is unitary, i.e. $\mathcal{S}_b \mathcal{S}_b^*=\mathcal{S}_b^*\mathcal{S}_b=I$ (see Theorem 7.32 in \cite{p1} in ${\mathbb R}^2$; exactly same argument applies in ${\mathbb R}^3$).
\end{definition} 
Next we assume that inside the anisotropic material $D$ there is a defect  (possibly anisotropic and/or absorbing) occupying the subregion $D_0$ such that $\overline{D}_0 \subset D$ having piecewise smooth boundary $\partial D_0$ (see Figure  \ref{geoexample1}). Note that $D_0$ can be of multiple components  with connected complement. We denote by $\tilde A_0$ and $\tilde n_0$ the material properties of the medium in $D_0$.  We further assume that the symmetric matrix-valued function $\tilde A_0$ is such that $\tilde A_0\in  C^{1} \left( D_0, \C^{m \times m} \right)$, $\xi\cdot  \Re(\tilde A_0(x)) \xi\geq \alpha_0|\xi|^2$, $\xi\cdot  \Im(\tilde A_0(x)) \xi \leq 0$ for all $\xi\in{\mathbb C}^m$ and for all $x\in D_0$, whereas the scalar-valued function $\tilde n_0$ is such that $\tilde n_0 \in C^1(D)$, $\Re(\tilde n_0(x))\geq c_0>0$  and $\Im(\tilde n_0(x))\geq 0$ for  all $x\in D_0$. Let us denote by $A_0$ and $n_0$ the extensions 
 $$A_0(x):=\left\{\begin{array}{rrcll} \tilde A_0(x) & x\in D_0\;\;\; \\
 A \;\;\; & \;\;\; x\in{\mathbb R}^m\setminus \overline{D_0}
 \end{array}\right.  \qquad \qquad  n_0(x):=\left\{\begin{array}{rrcll} \tilde n_0(x) & \;\;\; x\in D_0\;\;\; \\
 n \;\;\; & x\in{\mathbb R}^m\setminus \overline{D_0}
 \end{array}\right. .$$ 
Obviously,  $A_0(x)$ and $n_0(x)$ are such that $A-A_0$ and $n-n_0$ are supported on $\overline {D}_0$.  Notice that a specific case of a defect is a void with  $\tilde A_0=I$ and $\tilde n_0=1$.  The scattering problem for the anisotropic media with the defective region $D_0$ now reads: find $u_0 \in H^1_{loc}(\R^m)$ with $u_0 =u_0^s + e^{ikx \cdot d}$ such that
\begin{eqnarray}
\grad \cdot A_0(x) \grad u_0 +k^2 n_0(x) u_0=0 \,  &\textrm{ in }& \,  \R^m  \label{defectp1}\\
\lim\limits_{r \rightarrow \infty} r^{\frac{m-1}{2}} \left( \frac{\partial u_0^s}{\partial r} -ik u_0^s \right)=0 \label{defectp4}
\end{eqnarray}
where again the radiation condition (\ref{bgp3}) is satisfied uniformly with respect to $\hat x=x/|x|$. Once again we recall that across the interfaces $\partial D$ and $\partial D_1$ we have that
$$\frac{\partial u^-}{\partial \nu_{A}} =\frac{\partial u^+}{\partial \nu} \quad \textrm{ on } \partial D \quad  \quad \frac{\partial u^-}{\partial \nu_{A_0}} =\frac{\partial u^+}{\partial \nu_A} \quad \textrm{ on } \partial D_0.$$ 
Similarly since $u^s_0$ is a radiating solution to the Helmholtz equation in $\R^m \setminus \overline{D}$, we have that its corresponding far field pattern $u_0^{\infty}(\hat{x}, d)$ is given by \eqref{ffpdef} where  $u^s_b$ is replaced with $u^s_0$. The far field operator  $F_0: L^2(\mathbb{S}) \longmapsto  L^2(\mathbb{S})$ for the defective anisotropic media is now defined by 
$$(F_0 g)(\hat{x}):=\int\limits_{\mathbb{S}} u_0^{\infty}(\hat{x}, d ) g(d) \, ds(d) \, \,  \text{ where } \, \,  g(d) \in L^2(\mathbb{S}).$$
The {\it inverse problem } we consider here is to determine the support of $D_0$ from a knowledge of $F_0$, i.e. from a knowledge of the  measured far field pattern $u_0^{\infty}(\hat{x}, d)$ for all $d,\hat x\in{\mathbb S}$, provided that $A$, $n$ and $D$ are known.

One can see that,  if we take the incident field in \eqref{defectp1}-\eqref{defectp4} to be $u_b(\cdot , d)=u^s_b + e^{ikx \cdot d}$ then the resulting scattered field $u^s=u_0^s-u_b^s$ is due to the defect $D_0$. Note  that the scattered field $u^s$ due to the incident field $u_b(\cdot , d)=u^s_b + e^{ikx \cdot d}$ satisfies the source problem 
\begin{eqnarray}
\grad \cdot A_0 \grad u^s +k^2 n_0 u^s=\grad \cdot (A-A_0)  \grad u_b +k^2(n-n_0)u_b \quad   \textrm{ in } \,  \R^m. \label{void}
\end{eqnarray} 
together with the Sommerfeld radiation condition, which coincides with the equation for $u_0^s-u_b^s$ by linearity and (\ref{bgp1}) and (\ref{defectp1}). Therefore the relative far-field operator associated with the scattered field due to the defect is given by
$$(F g)(\hat{x}):=\int\limits_{\mathbb{S}}\big[ u_0^{\infty}(\hat{x}, d )-u_b^{\infty}(\hat{x}, d )\big]  g(d) \, ds(d) \, \,  \text{ where } \, \,  g(d) \in L^2(\mathbb{S}),$$
which is $F=F_0-F_b$. Note that $F_0$ is what we measure and $F_b$ is computable since $A$, $n$ and $D$ are known, hence we can assume that we know $F$.
\begin{remark}
{\em The smoothness of the coefficients $A_0$, $A$, $n$ and $n_0$ in our analysis can be relaxed to e.g. to  Lipshitz continuous or as regular as it is needed to apply unique continuation to the solution of the direct scattering problem.}
\end{remark}
\section{Factorization of the Far Field Operator }
\noindent Our goal in the current section is to construct a factorization of  the relative far field operator $F=F_0-F_b$ in such a way as to use the factorization method in  \cite{armin}, \cite{kirschbook}, in order to develop a range test for the support $D_0$ of the defect in terms of the measured far field operator. To this end motivated by the expression (\ref{void}) for the  scattered field due to the defect, we consider the problem of finding  $u \in H^1_{loc}(\R^m)$ for a given $v \in H^1(D_0)$ such that 
\begin{eqnarray}
&&\grad \cdot A_0 \grad u +k^2 n_0 u=\grad \cdot (A-A_0)  \grad v +k^2(n-n_0)v \quad   \textrm{ in } \,  \R^m \label{sourcep1} \\
&&\hspace{0cm} \lim\limits_{r \rightarrow \infty} r^{\frac{m-1}{2}} \left( \frac{\partial u}{\partial r} -ik u \right)=0. \nonumber
\end{eqnarray}
At this point let us recall  the exterior Dirichlet-to-Neumann map ${\mathbb T}_k: H^{1/2}(\partial B_R) \mapsto H^{-1/2}(\partial B_R)$ given by ${\mathbb T}_k f=\frac{\partial \varphi}{\partial \nu}$ on $\partial B_R$ where
\begin{eqnarray*}
&\Delta \varphi +k^2 \varphi=0\qquad  \quad \text{in } \R^m \setminus \overline{B}_R & \\
&\qquad \varphi=f\qquad \qquad  \quad \text{on } \partial B_R&\\
& \lim\limits_{r \rightarrow \infty} r^{\frac{m-1}{2}} \left( \frac{\partial \varphi}{\partial r} -ik \varphi \right)=0&
\end{eqnarray*}
with $B_R=\{ x \in \R^m : |x| < R\}$. With help of Dirichlet-to-Neumann operator we can write (\ref{sourcep1}) in the following  equivalent variational form: find $u\in H^1(B_R)$ such that  
 \begin{eqnarray}
\hspace*{-1.5cm}  &&\int\limits_{B_R} A_0\grad u \cdot  \grad \overline{\varphi} -k^2 n_0u  \overline{\varphi}\, dx -\int\limits_{\partial B_R} {\mathbb T}_k  u \,\overline{\varphi}\, ds \nonumber\\
&&=\int\limits_{D_0} (A-A_0)\grad v \cdot  \grad \overline{\varphi} -k^2 (n-n_0)v \overline{\varphi}\, dx, \quad \forall \varphi \in H^1(B_R),  \label{hadjointform0}
 \end{eqnarray}
which will be used frequently in what follows. It is standard  to shown that the above problem is well-posed, and furthermore if $v=u_b|_{D_0}$ we see that the scattered field $u^s=u^s_0-u^s_b$ (where $u^s_b$ and $u^s_0$ are the scattered fields for  \eqref{bgp1}-\eqref{bgp3} and  \eqref{defectp1}-\eqref{defectp4}, respectively) must coincide with $u$ given by \eqref{sourcep1}.
We now define the source-to-far field pattern operator as 
$$G: H^1(D_0) \longmapsto L^2(\mathbb{S}) \quad  \text{given by} \quad Gv:=u^{\infty}.$$
In addition, let us define
$$v^b_g(x):=\int\limits_{\mathbb{S}} u_b(x,d) g(d) \, ds(d), \qquad \qquad g\in L^2(\mathbb {S}),\;\; x\in{\mathbb R}^m$$
where $u_b (x,d)=u^s_b(x,d) + e^{ikx \cdot d}$ solves \eqref{bgp1}-\eqref{bgp3}  and consider the bounded linear operator 
 $$H: L^2(\mathbb{S}) \mapsto  H^1(D_0), \quad \text {defined by } \quad Hg:= v^b_g|_{D_0}.$$
Obviously $F=GH$. To further factorize the operator $F$ we first  need to compute the adjoint  $H^*: H^1(D_0)\to L^2(\mathbb{S})$ of  the operator $H$ defined above.  
 \begin{lemma}\label{hadjoint}
 The operator $H^*: H^1(D_0  ) \longmapsto L^2(\mathbb{S})$ is given by 
 $$-\gamma_m H^* v= \mathcal{S}_b^* \tilde{v}^{\infty}$$
 where $\tilde{v}^{\infty}$ is the far field pattern of the radiating field  $\tilde{v} \in H^1_{loc}(\R^m)$ satisfying 
 \begin{eqnarray}
\hspace*{-1.5cm}  -\int\limits_{B_R} A\grad \tilde{v} \cdot  \grad \overline{\varphi} -k^2 n\tilde{v}  \overline{\varphi}\, dx +\int\limits_{\partial B_R} \overline{\varphi} {\mathbb T}_k  \tilde{v}\, ds =(v,\varphi)_{H^1(D_0)}, \quad \forall \varphi \in H^1(B_R)  \label{hadjointform}
 \end{eqnarray}
 \end{lemma}
\begin{proof}
Let $v \in  H^1(D_0)$ be given then we can construct a unique radiating field  $\tilde{v} \in H^1_{loc}(\R^m)$ that satisfies \eqref{hadjointform} (see Chapter 5 of \cite{p1}). Now we have that integration by parts gives 
\begin{eqnarray*}
&&(H^* v,g)_{L^2(\mathbb{S})} = (v,Hg)_{H^1(D_0)} \\
					    &&\hspace{1cm}= -\int\limits_{B_R}\left( A\grad \tilde{v} \cdot  \grad \overline{v_g^b} -k^2 n\tilde{v}  \overline{v_g^b}\right)\, dx +\int\limits_{\partial B_R} \overline{v_g^b} {\mathbb T}_k  \tilde{v}\, ds\\
					    &&\hspace{1cm}=\int\limits_{\partial B_R}\left( \overline{v_g^b} \frac{\partial  \tilde{v}}{\partial \nu} -  \tilde{v}  \frac{\partial \overline{v_g^b} }{\partial \nu}\right)\, ds+\int\limits_{B_R} \tilde{v} (\grad \cdot A \grad \overline{v_g^b} +k^2 n \overline{v_g^b}) \, dx  
\end{eqnarray*}
where we recall that  ${\displaystyle (v_g^b)(x)=\int_{\mathbb{S}} \left( u^s_b(x,d) + e^{ikx \cdot d} \right) g(d) \, ds(d) }$ for all of $x \in \R^m$. Using that the matrix $A$ is real symmetric along with $\Im \{n(x) \}=0$ and that $\grad \cdot A \grad {v_g^b} +k^2 n {v_g^b}=0$ in $\R^m$ gives that the integral over $B_R$ is zero. Now by using the definition of $v_g^b$ and changing the order of integration we have that  
\begin{eqnarray}
&&(H^* v,g)_{L^2(\mathbb{S})}  = \int\limits_{\partial B_R}\left( \overline{v_g^b} \frac{\partial  \tilde{v}}{\partial \nu} -  \tilde{v}  \frac{\partial \overline{v_g^b} }{\partial \nu}\right)\, ds_x\nonumber \\
					     &&\hspace{1cm}= \int\limits_{\mathbb{S}} \overline{g(d)} \left[ \, \int\limits_{\partial B_R} \left( \frac{\partial  \tilde{v}}{\partial \nu} e^{-ikx \cdot d} -  \tilde{v}  \frac{\partial e^{-ikx \cdot d} }{\partial \nu}\, ds_x\right)\right] \, ds(d)\nonumber \\
					    &&\hspace{1cm}+ \int\limits_{\mathbb{S}} \overline{g(d)} \left[ \, \int\limits_{\partial B_R} \left( \overline{ u^s_b(x,d) } \frac{\partial  \tilde{v}}{\partial \nu}  -  \tilde{v}  \frac{\partial \overline{u^s_b(x,d)} }{\partial \nu}\right)\, ds_x \right] \, ds(d).\label{sss}
\end{eqnarray}
We notice that \eqref{ffpdef} gives that 
\begin{eqnarray}
\hspace*{-1cm}\int\limits_{\mathbb{S}} \overline{g(d)} \left[ \, \int\limits_{\partial B_R} \left( \frac{\partial  \tilde{v}}{\partial \nu} e^{-ikx \cdot d} -  \tilde{v}  \frac{\partial e^{-ikx \cdot d} }{\partial \nu}\right)\, ds_x \right] \, ds(d) = -\frac{1}{\gamma_m}(\tilde{v}^{\infty},g)_{L^2(\mathbb{S})}. \label{hadjointint1}
\end{eqnarray}
Using the asymptotic behavior of a  radiating solution to Helmholtz equation and its derivative   (see \cite{p1} for the case of ${\R}^2$ and \cite{coltonkress} for the case of ${\R}^3$) and  letting $R \rightarrow \infty$ the second integral in (\ref{sss}) becomes
$$\int\limits_{\mathbb{S}} \overline{g(d)} \left[ 2ik  \int\limits_{\mathbb{S}} \tilde{v}^{\infty}(\hat{x}) \overline{u_b^{\infty}(\hat{x}, d )} \, ds(\hat{x}) \right] \, ds(d)$$
Using the reciprocity identity $u_b^{\infty}(\hat{x}, d )=u_b^{\infty}( -d, -\hat{x} )$ (see Theorem 7.30 in \cite{p1}) and making the change of variables $\hat{x} \longmapsto -\hat{x}$ we obtain
\begin{eqnarray}
\int\limits_{\mathbb{S}} \overline{g(d)} \left[ 2ik  \int\limits_{\mathbb{S}} \tilde{v}^{\infty}(-\hat{x}) \overline{u_b^{\infty}( -d, \hat{x} )} \, ds(\hat{x}) \right] \, ds(d)=2ik(F^*_b \tilde{v}^{\infty}, g)_{L^2(\mathbb{S})}.\label{hadjointint2}
\end{eqnarray}
Finally,  combining \eqref{hadjointint1} and \eqref{hadjointint2} we have that 
$$H^* v= \left(-\frac{1}{\gamma_m}I +2ik F^*_b \right) \tilde{v}^{\infty}$$
giving the result by multiplying by $-\gamma_m$ and by Definition \ref{deft}. 
\end{proof}
Now for any given $\phi \in  H^1(D_0)$ we can construct  a function $w_{\phi} \in  H^1_{loc}(\R^m)$ that satisfies
\begin{eqnarray}
&&\grad \cdot A\grad w_\phi +k^2 n w_\phi=\grad \cdot (A-A_0)  \grad \phi +k^2(n-n_0)\phi \,  \textrm{ in } \,  \R^m \label{sourcepw}\\
&&\lim\limits_{r \rightarrow \infty} r^{\frac{m-1}{2}} \left( \frac{\partial w_\phi}{\partial r} -ik w_\phi \right)=0 \nonumber 
\end{eqnarray}
and then let $u \in H^1_{loc}(\R^m)$ be the unique solution to \eqref{sourcep1} for a given $v \in H^1(D_0)$. Now by letting $\phi= v+u|_{D_0}$ and the corresponding  $w:=w_{\phi}$,  we observe that this $w$ satisfies the variational problem 
\begin{eqnarray*}
&&\hspace*{-1cm}-\int\limits_{B_R} A\grad w \cdot  \grad \overline{\varphi} -k^2 nw \overline{\varphi}\, dx + \int\limits_{\partial B_R} \overline{\varphi} {\mathbb T}_k w\, ds  = - \int\limits_{D_0} (A-A_0) \grad (v+u) \cdot \grad \overline{ \varphi } \, dx \\
 &&\hspace{2cm} +k^2  \int\limits_{D_0} (n-n_0)(v+u) \overline{ \varphi } \, dx \qquad \forall \varphi \in H^1(B_R).
\end{eqnarray*}
Next by means of Riesz representation theorem, we define the bounded linear operator $T: H^1(D_0  ) \longmapsto H^1(D_0  )$  such that for all $\varphi \in H^1(D_0)$
\begin{equation} 
(T v, \varphi)_{ H^1(D_0  ) }= - \int\limits_{D_0} (A-A_0) \grad (v+u) \cdot \grad \overline{ \varphi }  \, dx  + \int\limits_{D_0} k^2 (n-n_0)(v+u) \overline{ \varphi } \, dx.   \,\,\,  \label{tk}
\end{equation}
Notice that the function $u$ defined by solving \eqref{sourcep1} satisfies
\begin{eqnarray}
\grad \cdot A\grad u +k^2 n u=\grad \cdot (A-A_0)  \grad (v+u) +k^2(n-n_0)(v+u)  \text{  in  }  \R^m  \label{sourcep2}
\end{eqnarray}
together with the Sommerfeld radiation condition, which gives that $u=w$ in $\R^m$ since \eqref{sourcepw} is well-posed. Therefore we conclude that $u^{\infty}=w^{\infty}$. Now by the definition of the operators $G$ we have that $u^{\infty}=Gv$ while using the definition of $H^*$ and  $T$ we have that $w^{\infty}=-\gamma_m\mathcal{S}_b H^*T v$. We now conclude that $Gv =-\gamma_m\mathcal{S}_b H^*T v$.  From the above analysis and the fact that $F=GH$ we have the following factorization. 

\begin{theorem}
The far field operator $F: L^2(\mathbb{S}) \longmapsto  L^2(\mathbb{S})$ associated with \eqref{sourcep1}   can be factorized as $F=-\gamma_m \mathcal{S}_bH^*TH$.
\end{theorem}

\section{The Factorization Method }
\noindent In this section we connect the support of the defect $D_0$ to the range of an operator defined by the measured far field operator based on the factorization method discussed in \cite{kirschbook} or  \cite{armin}. We make this connection by analyzing the factorization of the far field operator developed in the previous section.  Defining $\tilde F:=\gamma_m^{-1}\mathcal{S}_b^*F$, we recall from  the previous section that  we have the following factorization $\tilde F=-H^*TH$. Under appropriate assumptions on the operators $H$ and $T$ the factorization method states that   the range of the operators $H^*: H^1(D_0) \mapsto L^2(\mathbb{S})$ and $\tilde{F}_{\sharp}^{1/2} : L^2(\mathbb{S}) \mapsto L^2(\mathbb{S})$ coincide, where $\tilde{F}_{\sharp}= |\Re(\tilde{F})| +|\Im(\tilde{F})|$. 

To arrive at the above range test we use the abstract theorems proven in \cite{kirschbook} and \cite{armin} on the range identities. To this end, we recall that for a generic bounded linear operator $B:X \to Y$, where $X$ and $Y$ are Banach spaces, we define the real and imaginary part selfajoint operators by
$$\Re(B)=\frac{B+B^*}{2}  \quad \text{and}  \quad \Im(B)=\frac{B-B^*}{2i}.$$ Furthermore for a generic self-adjoint compact operator $B$ on a Hilbert space $U$, $|B|$ is defined in terms of the spectral decomposition as $|B|(x) = \sum |\lambda_j | (x, \psi_j) \psi_j$ for all $x\in U$ where  $(\lambda_j , \psi_j) \in \R \times U$ is the  orthonormal eigensystem of $B$. 
Now, let $X \subset U \subset X^*$ be a Gelfand triple with a Hilbert space $U$ and a reflexive Banach space $X$ such that the embedding is dense. Furthermore, let $Y$ be a second Hilbert space and let $\tilde{F}: Y \mapsto Y$, $H: Y \mapsto X$ and $T: X \mapsto X^*$ be linear bounded operators such that $\tilde{F}=H^*TH$.
\begin{theorem}(Theorem 2.15 in \cite{kirschbook}) \label{fmthm1}
Assume that
\begin{enumerate}
\item $H^*$ is compact with dense range.
\item There exists  $t \in [0, 2 \pi ]$ such that $\Re(e^{it}T)$ is the sum of a compact operator and a self adjoint coercive operator. 
\item $\Im(T)$ is compact and non-negative on the range $\mathcal{R}(H)$ of $H$.
\item $\Re(e^{it}T)$ is injective or $\Im(T)$ is strictly positive on the closure $\overline{ \mathcal{R}(H) }$.
\end{enumerate}
Then the operator $\tilde{F}_{\sharp}= |\Re(e^{it} \tilde{F})| +\Im(\tilde{F})$ is positive, and the range of the operators $H^*: X^* \mapsto Y$ and $\tilde{F}_{\sharp}^{1/2} : Y \mapsto Y$ coincide. 
\end{theorem}
\begin{theorem}(Theorem 2.1 in \cite{armin}) \label{fmthm2}
Assume that 
\begin{enumerate}
\item $H$ is compact and injective.
\item  $\Re(T)$ is the sum of a compact operator and a self adjoint coercive operator. 
\item $\Im(T)$ is non-negative on $X$.\\
\hspace*{-1.2cm} Moreover assume that either of the following is satisfied:
\item $T$ is injective. 
\item $\Im(T)$ is strictly positive on the (finite dimensional) null space of $\Re (T)$.
\end{enumerate}
Then the operator $\tilde{F}_{\sharp}= |\Re(\tilde{F})| +\Im(\tilde{F})$ is positive, and the range of the operators $H^*: X^* \mapsto Y$ and $\tilde{F}_{\sharp}^{1/2} : Y \mapsto Y$ coincide. 
\end{theorem}
We note that just as in the remark after Theorem 2.15 in \cite{kirschbook} we have that  if $\Im(T)$ is non-positive then both theorems hold for $\tilde{F}_{\sharp}= |\Re(\tilde{F})| -\Im(\tilde{F})$, hence in either case we can use $|\Re(\tilde{F})| +|\Im(\tilde{F})|$ in the range test.  

We dedicate this section to showing that $H$ and $T$ satisfy the necessary conditions to apply any of the above range tests. To this end, let's  define the interior transmission eigenvalue problem in the defective region $D_0$ as finding a pair $(w,v) \in H^1(D_0) \times H^1(D_0)$ such that for given $(f,h) \in H^{1/2}(\partial D_0) \times H^{-1/2}(\partial D_0)$ satisfies
\begin{eqnarray}
\grad \cdot A_0 \grad w +k^2 n_0 w=0 \,  &\textrm{ in }& \,  D_0 \label{fmitp1}\\
 \grad \cdot A \grad  v + k^2 n v=0  &\textrm{ in }& \,  D_0 \\
w-v= f  &\textrm{ on }& \partial D_0\\
  \frac{\partial w}{\partial \nu_{A_0}}-\frac{\partial v}{\partial \nu_{A}}=  h &\textrm{ on }& \partial D_0 \label{fmitp3}
\end{eqnarray}
\begin{definition} \label{trig} {\em The values of $k \in \R$  for which the homogeneous interior transmission problem,  i.e. \eqref{fmitp1}-\eqref{fmitp3} with  $(f,h)=(0,0)$, has nontrivial  solutions are called transmission eigenvalues for $D_0$.}
\end{definition}
The following results are know if $A_0=I$ and $n_0=1$. The proofs can be readily extended to the current case. We state the results and give the corresponding reference for the proof in the case of $A_0=I$ and $n_0=1$.
\begin{theorem}
 Assume that $A-\Re(A_0)$ is positive definite or negative definite. Then  \eqref{fmitp1}-\eqref{fmitp3} satisfies the Fredholm alternative, i.e if $k$ is not a transmission eigenvalue there exits a unique solution to  \eqref{fmitp1}-\eqref{fmitp3} that depends continuously on the data $(f,h)$.
\end{theorem}
\hspace*{-0.8cm}  See \cite{p1} for the proof.

\begin{theorem}
\begin{enumerate}
\item If $\Im(A_0)<0$ and/or $\Im (n_0)>0$ in $D_0$ then there are no real transmission eigenvalues.
\item Assume that $\Im(A_0)=0$ and $\Im(n_0)=0$. Then the set of real transmission eigenvalues is at most discrete with $+\infty$ as the only possible accumulation point provided:
\begin{enumerate}
\item $A-A_0$ is positive or negative definite  uniformly in $D_0$ and ${\displaystyle \int_{D_0} (n-n_0) \, dx \neq 0}$,
\item $A-A_0$ is positive or negative definite uniformly in $D_0$ and $n\equiv n_0$.
\end{enumerate}
\end{enumerate}
\end{theorem}
\hspace*{-0.8cm} See \cite{p1}, Chapter 6 for the proof of parts $(i)$ and  $(ii)(b)$, and \cite{Bon-Che-Had-2011} for the proof of part $(ii)(a)$. 

\bigskip

We call   $\mathbb{G}( \cdot ,\cdot)$ the Green's function of the background media, i.e.  $\mathbb{G}( \cdot ,z)\in H^1_{loc}(\R^m\setminus\{z\})$ which solves
\begin{eqnarray*}
&&\grad \cdot A\grad  \mathbb{G}( \cdot ,z) +k^2 n  \mathbb{G}( \cdot ,z)=-\delta (\, \cdot \, - z) \, \,  \textrm{ in } \,  \R^m\setminus\{z\} \\
&&\lim\limits_{r \rightarrow \infty} r^{\frac{m-1}{2}} \left( \frac{\partial  \mathbb{G}( \cdot ,z)}{\partial r} -ik  \mathbb{G}( \cdot ,z) \right)=0  
\end{eqnarray*}
Outside of the scattering object $D$ we have that, for a fixed $z \in \R^m$,  $\mathbb{G}( \cdot \, , z)$ is a radiating solution to Helmholtz equation in $\R^m \setminus \overline{B}_R$ for some $R$ sufficiently large. So we let $\mathbb{G}^{\infty}( \cdot \, , z) \in L^2(\mathbb{S})$ be the far field pattern of $\mathbb{G}( \cdot \, , z)$. 
\begin{theorem} \label{rangeh}
The operator $H^*: H^1(D_0  ) \longmapsto L^2(\mathbb{S})$ defined in Lemma \ref{hadjoint} satisfies the following: 
\begin{enumerate}
\item $H^*$ is compact with dense range (or in other words $H$ is compact and injective). 
\item $\mathcal{S}^*_b \mathbb{G}^{\infty}( \cdot ,z) \in \mathcal{R}(H^*)$ if and only if $z \in D_0$.
\end{enumerate}
\end{theorem}
\begin{proof}
{\it (i)} $H^*$ is compact due to the fact that the mapping $v \mapsto \tilde{v}$ is bounded from $H^1(D_0)$ to $H^1_{loc}(\R^m)$ and $\tilde{v} \mapsto  \tilde{v}^{\infty}$ is compact from $H^1_{loc}(\R^m)$ to $L^2(\mathbb{S})$. We have also used that the scattering operator is bounded. Now to prove that $H^*$ has dense range it is sufficient to prove that $H$ is injective. So assume that $g \in L^2(\mathbb{S})$ is such that $Hg=0$, then $v_g^b$ defined by ${\displaystyle v_g^b:=\int_{\mathbb{S}} u_b(x,d) g(d) \, ds(d)}$ for all $x \in \R^m$ therefore we have that $v_g^b=0$ in $D_0$. Therefore since $v^b_g$ satisfies 
$$ \grad \cdot A \grad  v^b_g + k^2 n v^b_g=0  \quad \text{in} \quad  \,  \R^m$$
we have that,  by a unique continuation argument, $v^b_g=0$ in any large ball $B_R$ of  arbitrary radius $R$.  But  $v^b_g=u^s_g+u^i_g$ where ${\displaystyle u^s_g=\int_{\mathbb{S}} u^s_b(x,d) g(d) \, ds(d)}$ and ${\displaystyle u^i_g=\int_{\mathbb{S}} e^{ikx\cdot d} g(d) \, ds(d)}$. We now observe that $u^s_g$ is a radiating solution to the Helmholtz equation whereas $u^i_g$ is an entire solution to the Helmholtz equation. Hence $u^s_g=u^i_g=0$ in $\R^m$ which implies $g=0$ in $L^2(\mathbb{S})$.

{\it (ii)} Let $z \in \R^m \setminus \overline{D}_0$ and assume that there is some $v \in H^1(D_0)$ such that $H^* v=\mathcal{S}^*_b \mathbb{G}^{\infty}( \cdot ,z)$. We can then conclude by the definition of $H^*$ that there is a $\tilde{v} \in H^1_{loc}(\R^m)$ satisfying \eqref{hadjointform} and therefore $\grad \cdot A \grad  \tilde{v}+k^2 n  \tilde{v}=0$ in $\R^m \setminus \overline{D}_0$ and $\tilde{v}^{\infty}=\mathbb{G}^{\infty}( \cdot , z)$. Therefore Rellich's lemma and unique continuation gives that $\tilde{v}= \mathbb{G}( \cdot ,z)$ in $\R^m \setminus (\overline{D}_0 \cup \{z\})$, which is a contradiction since $\mathbb{G}( \cdot ,z ) \notin H^1(B_r(z))$ and $  \tilde{v} \in H^1(B_r(z))$, for any disk $B_r(z)$ centered at $z$ of radius $r>0$. 

Now let $z \in D_0$ then we have that $\mathbb{G}( \cdot ,z ) \in H^1_{loc}(\R^m \setminus \overline{D}_0)$. Since $k$ is not a transmission eigenvalue in $D_0$ we can construct $(w_z,v_z)$ that solve the interior transmission problem \eqref{fmitp1}-\eqref{fmitp3} with $(f,h)=\left( \mathbb{G}( \cdot \, , z) \, ,\,  \displaystyle{\frac{\partial}{\partial \nu_{A}}} \mathbb{G}( \cdot \, , z)  \right)$. Now let 
$$u_z :=\left\{ \begin{array}{rcl} w_z-v_z & \mbox{in} & D_0 \\   \mathbb{G}( \cdot \, , z) & \mbox{in} &\R^m \setminus \overline{D}_0  \end{array} \right.$$
therefore we have that $u_z \in H^1_{loc}(\R^m)$ with $u^{\infty}_z=\mathbb{G}^{\infty}( \cdot ,z)$ such that  
\begin{eqnarray*}
&&\grad \cdot A\grad u_z +k^2 n u_z=\grad \cdot (A-A_0)  \grad w_z +k^2(n-n_0) w_z  \,  \textrm{ in } \,  \R^m. 
\end{eqnarray*}
The latter  implies that for all $\varphi \in H^1(B_R)$ 
 \begin{eqnarray}
&&\hspace*{-1.5cm} -\int\limits_{B_R} A\grad u_z \cdot  \grad \overline{\varphi} -k^2 n u_z  \overline{\varphi}\, dx +\int\limits_{\partial B_R} \overline{\varphi} {\mathbb T}_k u_z \, ds =- \int\limits_{D_0} (A-A_0) \grad w_z \cdot \grad \overline{ \varphi }  \, dx  \nonumber  \\  
&&\hspace{3.5cm} + \int\limits_{D_0} k^2 (n-n_0) w_z \overline{ \varphi } \, dx. \label{theta}
 \end{eqnarray}
Let  $\theta_z \in H^1(D_0)$ be defined from the right hand side of (\ref{theta}) by means of the Riesz representation theorem, hence we have  
 \begin{eqnarray*}
-\int\limits_{B_R} A\grad u_z \cdot  \grad \overline{\varphi} -k^2 n u_z  \overline{\varphi}\, dx +\int\limits_{\partial B_R} \overline{\varphi} {\mathbb T}_k u_z \, ds =(\theta_z,\varphi)_{H^1(D_0)}
 \end{eqnarray*}
Thus we now  conclude that $-\gamma_m H^* \theta_z= \mathcal{S}_b^*\mathbb{G}^{\infty}( \cdot ,z)$ by the definition of $H^*$ giving the result. 
\end{proof}

Next we analyze the properties of the middle operator $T$ defined by \eqref{tk}.

\begin{theorem} \label{tinjective}
The operator $T : H^1(D_0) \longmapsto  H^1(D_0)$ is injective provided that either one of the following  conditions are satisfied:
\begin{enumerate}
\item $\Im(A_0) <0$ in $D_0$ and $\displaystyle{\int_D(n-n_0)\, dx \neq 0}$.
\item $\Im(A_0)\leq 0$ and $\Im(n_0)>0$ in $D_0$. 
\item $\Im(A_0) =0$, $\Im(n_0) =0$ in $D_0$  and either  $A-A_0>0$ and $n-n_0<0$ or $A-A_0<0$ and $n-n_0>0$ in $D_0$.
\end{enumerate}
\end{theorem}
\begin{proof}
Assume that $Tv=0$, therefore $u \in H^1_{loc}(\R^m)$ defined by solving \eqref{sourcep1} satisfies for all $\varphi \in H^1(B_R)$ 
\begin{eqnarray*}
-\int\limits_{B_R} A\grad u \cdot  \grad \overline{\varphi} -k^2 n u \overline{\varphi}\, dx + \int\limits_{\partial B_R} \overline{\varphi} {\mathbb T}_k u \, ds  = (Tv, \varphi)_{H^1(D_0)}=0. 
\end{eqnarray*}
which implies that $u=0$ and therefore we have that for all $\varphi \in H^1(D_0)$
 \begin{eqnarray}  
 \int\limits_{D_0} (A-A_0) \grad v \cdot \grad \overline{ \varphi }  - k^2 (n-n_0)v \overline{ \varphi } \, dx =0 \label{injectproof}
\end{eqnarray}
Letting $\varphi:=v$, parts (i) and (ii) of the proof follow by taking the imaginary part of \eqref{injectproof} (note that $\Im(n_0)\geq 0$) whereas part (iii) is obvious from the assumptions. 
\end{proof}

\begin{theorem} \label{tkanalysis}
The imaginary part of the operator $T : H^1(D_0) \longmapsto  H^1(D_0)$ satisfies the following properties:
\begin{enumerate}
\item $(\Im (T) v,v)_{H^1(D_0)}\leq 0$. 
\item If $k$ is not a transmission eigenvalue for $D_0$ then $(\Im (T) v,v)_{H^1(D_0)}<0$ for  $v\in \overline{ \mathcal{R}(H) }$. 
\item If  $\Im(A_0) = 0$  then $\Im(T)$ is compact.
\end{enumerate}
\end{theorem}

\begin{proof}
{\it (i)} Recall that for any $v_j \in H^1(D_0)$ there is a unique $u_j \in H^1_{loc}(\R^m)$ that is a solution to \eqref{sourcep1}. Now we let $\phi_j=v_j+u_j$, therefore using \eqref{tk} we have that 
\begin{eqnarray*}
&&\hspace*{-1.5cm}(T v_1, v_2)_{ H^1(D_0  ) }= - \int\limits_{D_0} (A-A_0) \grad \phi_1 \cdot \grad \overline{ (\phi_2-u_2) }-k^2 (n-n_0)\phi_1 \overline{ (\phi_2-u_2 )} \, dx\\
&&\hspace*{1.5cm}= - \int\limits_{D_0} (A-A_0) \grad \phi_1 \cdot \grad \overline{ \phi_2 }-k^2 (n-n_0)\phi_1 \overline{ \phi_2} \, dx\\
&&\hspace*{1.5cm}+ \int\limits_{D_0} (A-A_0) \grad \phi_1 \cdot \grad \overline{ u_2 }-k^2 (n-n_0)\phi_1 \overline{ u_2} \, dx.
\end{eqnarray*}
Now using that 
\begin{eqnarray*}
\grad \cdot A\grad u_1 +k^2 n u_1=\grad \cdot (A-A_0)  \grad \phi_1 +k^2(n-n_0)\phi_1 \quad \text{ in } \quad    \R^m 
\end{eqnarray*}
multiplying by $\overline{u_2}$ and integrating by parts over $B_R$ such that $D\subset B_R$ we have that
\begin{eqnarray*}
&&\hspace*{-1.5cm}-\int\limits_{B_R} A\grad u_1 \cdot  \grad \overline{u_2} -k^2 n u_1 \overline{u_2}\, dx +\int\limits_{\partial B_R} \overline{u_2}  \frac{\partial u_1}{\partial \nu}\, ds  = - \int\limits_{D_0} (A-A_0) \grad \phi_1 \cdot \grad \overline{ u_2 } \, dx \\
&&\hspace*{5cm} +k^2  \int\limits_{D_0} (n-n_0) \phi_1 \overline{ u_2 } \, dx. 
\end{eqnarray*}
This gives that 
\begin{eqnarray}
&&(T v_1, v_2)_{ H^1(D_0  ) }= - \int\limits_{D_0} (A-A_0) \grad \phi_1 \cdot \grad \overline{ \phi_2 }-k^2 (n-n_0)\phi_1 \overline{ \phi_2} \, dx \nonumber \\
&&\hspace*{2cm}+\int\limits_{B_R} A\grad u_1 \cdot  \grad \overline{u_2} -k^2 n u_1 \overline{u_2}\, dx -\int\limits_{\partial B_R} \overline{u_2}  \frac{\partial u_1}{\partial \nu}\, ds.    \label{vart1}
\end{eqnarray}
Now taking the imaginary part of (\ref{vart1}) where we substitute $v_2$ by $v_1$,  using the fact that  $A$ and $A_0$ are  symmetric matrices, $A$ and $n$ are  real  valued and letting $R \rightarrow \infty$ we obtain
 \begin{equation}\label{im0}
 (\Im (T) v_1, v_1)_{ H^1(D_0  ) } =  \int\limits_{D_0} \Im(A_0) |\grad \phi_1|^2-k^2 \Im(n_0) |\phi_1|^2 \, dx- k \int\limits_{\mathbb{S} }|u^{\infty}_1|^2 \, ds(\hat{x}) 
 \end{equation}
where $u_1^\infty$ is defined by the asymptotic expansion of the radiating field  $u_1$ 
$$u_1(x)=\frac{e^{ikr}}{r^{\frac{m-1}{2}}}u_1^\infty(\hat x)+O\left(r^{-\frac{m+1}{2}}\right), \qquad r=|x|, \;\; \hat x=x/|x|,$$ (see \cite{coltonkress} in ${\mathbb R}^3$ and \cite{p1} in ${\mathbb R}^2$),
which gives that $\Im (T)$ is non-positive.  \\

{\it (ii)} Now let $v \in \overline{\mathcal{R}(H)}$ and assume that $  (\Im (T) v, v)_{ H^1(D_0  ) }= 0$. Then there is a sequence  $v_\ell \in {\mathcal{R}(H)}$ such that $v_\ell \rightarrow v$ in $H^1(D_0)$,  and let $u_\ell  \in H^1_{loc}(\R^m)$ be the sequence of the corresponding solutions  of (\ref{sourcep1}). Since $u_\ell$ is bounded in $H^1_{loc}(\R^m)$ by the well-posedness  of (\ref{sourcep1}), we can conclude that $u_\ell  \weakc u$ weakly in $H^1_{loc}(\R^m)$ which implies that 
$$\lim_{\ell\to \infty} \int\limits_{\R^m} A_0\grad u_\ell \cdot  \grad \overline{\varphi} -k^2 n_0u_\ell  \overline{\varphi}\, dx=\int\limits_{\R^m} A_0\grad u\cdot  \grad \overline{\varphi} -k^2 n_0u \overline{\varphi}\, dx, \qquad \varphi\in H^1(\R^m).$$
Hence, this limit $u$ is a week solution of
\begin{eqnarray*}
&&\grad \cdot A_0 \grad u +k^2 n_0 u=\grad \cdot (A-A_0)  \grad v +k^2(n-n_0)v \quad   \textrm{ in } \, \R^m  \\
&& \grad \cdot A \grad v +k^2 n v =0 \quad   \textrm{ in } \,  D_0.
\end{eqnarray*}
Furthermore,  since $ (\Im (T) v_\ell, v_\ell)_{ H^1(D_0  ) } \rightarrow 0$, from (\ref{im0}) we conclude that $u^{\infty}=0$ whence by Rellich' lemma and  unique continuation $u$  is zero outside of $D_0$. So we have that  $u^+ =0$ and $\frac{\partial u^+}{\partial \nu_A}=0$ on $\partial D_0$ therefore the pair $(u+v,v)$ are transmission eigenfunctions for  $D_0$ but since $k$ is not a transmission eigenvalue we have that $v=0$. 

{\it (iii)} If $\Im(A_0) = 0$  then
$$  (-\Im (T) v_1, v_2)_{ H^1(D_0  ) } = \int\limits_{D_0} k^2 \Im(n_0) \phi_1 \overline{ \phi_2} \, dx+k \int\limits_{\mathbb{S} }u^{\infty}_1 \overline {u^{\infty}_2} \, ds(\hat{x}), $$ 
now using that the mapping $v \mapsto \tilde{v}$ is bounded from $H^1(D_0)$ to $H^1_{loc}(\R^m)$ and $\tilde{v} \mapsto  \tilde{v}^{\infty}$ is compact from $H^1_{loc}(\R^m)$ to $L^2(\mathbb{S})$, we can conclude that the second term in the variational form given above is compact. Furthermore from the fact that $H^1(D_0)$ is compactly embedded in $L^2(D_0)$, we can finally conclude that $\Im(T)$ is compact.
\end{proof}
\begin{theorem} \label{rtkanalysis}
The real part of the operator $T$ satisfies the following property:
\begin{enumerate}
\item If $\Re(A_0)-A$ is positive definite in $D_0$ then $\Re(T)$ is the sum of a compact operator and a self-adjoint coercive operator.
\item  If $(A-\Re(A_0)-\alpha |\Im(A_0)|)>0$ uniformly in $D_0$ and $(\Re(A_0)-\frac{1}{\alpha} |\Im(A_0)|)\geq 0$ for some constant $\alpha>0$  then $-\Re(T)$ is the sum of a compact operator and a self-adjoint coercive  operator.
\end{enumerate}
\end{theorem}
\begin{proof}
(i) Assume first that $\Re (A_0)-A$ is positive definite. Now by using the variational form \eqref{vart1} for $T$ and the Dirichlet to Neumann operator ${\mathbb T}_k$ we have that 
\begin{eqnarray}
&&(T v_1, v_2)_{ H^1(D_0  ) }= - \int\limits_{D_0} (A-A_0) \grad \phi_1 \cdot \grad \overline{ \phi_2 }-k^2 (n-n_0)\phi_1 \overline{ \phi_2} \, dx\label{ttt} \\
&&\hspace*{2cm}+\int\limits_{B_R} A\grad u_1 \cdot  \grad \overline{u_2} -k^2 n u_1 \overline{u_2}\, dx -\int\limits_{\partial B_R} \overline{u_2}  {\mathbb T}_k u_1ds.    
\end{eqnarray}
Now define the bounded linear operators $S$ and $K: H^1(D_0) \longmapsto  H^1(D_0)$ by the Riesz representation theorem such that 
\begin{eqnarray*}
&&(S v_1, v_2)_{ H^1(D_0  ) }= \int\limits_{D_0} (A_0-A) \grad \phi_1 \cdot \grad \overline{ \phi_2 }+\phi_1 \overline{ \phi_2} \, dx \\ 
&&\hspace*{3.5cm}+\int\limits_{B_R} A\grad u_1 \cdot  \grad \overline{u_2} + u_1 \overline{u_2}\, dx -\int\limits_{\partial B_R} \overline{u_2}  {\mathbb T}_k u_1\, ds \\  
&&\hspace*{-1cm}(-K v_1, v_2)_{ H^1(D_0  ) }=k^2 \int\limits_{D_0} (n-n_0)\phi_1 \overline{ \phi_2} \, dx +\int\limits_{D_0} \phi_1 \overline{ \phi_2} \, dx +\int\limits_{B_R}  (k^2n+1) u_1 \overline{u_2}\, dx. 
\end{eqnarray*}

By the definition of $T$ we have that $T=S+K$. By the compact embedding of $H^1(D_0)$ into $L^2(D_0)$ and $H^1(B_R)$ into $L^2(B_R)$ we have that $K$ is a compact operator which implies that $\Re(K)$ is also compact. We now show that $\Re(S)$ is self-adjoint and coercive on $H^1(D_0)$. Notice that since $A$ is a real symmetric matrix we have that
\begin{eqnarray*}
&&(\Re(S) v_1, v_2)_{ H^1(D_0  ) }=  \int\limits_{D_0} (\Re (A_0)-A) \grad \phi_1 \cdot \grad \overline{ \phi_2 }+\phi_1 \overline{ \phi_2} \, dx \\ 
&&\hspace*{3cm}+\int\limits_{B_R} A\grad u_1 \cdot  \grad \overline{u_2} + u_1 \overline{u_2}\, dx -\int\limits_{\partial B_R} \overline{u_2}  \Re({\mathbb T}_k) u_1\, ds   
\end{eqnarray*}
which gives that $\Re(S)$ is self-adjoint. To prove coercivity we write
\begin{eqnarray*}
&&(\Re(S) v_1, v_1)_{ H^1(D_0  ) }=  \int\limits_{D_0} (\Re (A_0)-A) |\grad(v_1+u_1)|^2+|(v_1+u_1)|^2 \, dx \\ 
&&\hspace*{3cm}+\int\limits_{B_R} A|\grad u_1|^2 + |u_1|^2\, dx -\int\limits_{\partial B_R} \overline{u_1}  \Re({\mathbb T}_k) u_1\, ds.   
\end{eqnarray*}
Using the fact that the real part of the Dirichlet to Neumann  $\Re({\mathbb T}_k)$ is non-positive (see e.g. \cite{dtnref} in ${\mathbb R}^3$) we obtain that 
$$(\Re(S) v_1, v_1)_{ H^1(D_0  ) }\geq \alpha \|v_1\|^2_{ H^1(D_0  )}$$ 
from a contradiction argument, namely by considering a sequence $v^{n}\in H^1(D_0)$ and the corresponding $u_n$ such that $\|v^{n}\|_{H^1(D_0)}=1$  for which $(\Re(S) v^{n}, v^{n})_{H^1(D_0)}\to 0$ we arrive at the  contradiction that $v^{n}\to 0$ in $H^1(D_0)$. This proves the claim when $\Re(A_0)-A$ is positive definite. \\

(ii) We now assume that $A-\Re(A_0)$ is positive definite. Unfortunately due to incompatible signs for $A-\Re(A_0)$ and the real part of the Dirichlet-to-Neumann operator we can not work with  (\ref{ttt}) for the operator $T$. To derive an appropriate expression for $T$,   we use \eqref{tk} and  letting  $\phi_j=v_j+u_j$, we arrive at
\begin{eqnarray*}
&&(T v_1, v_2)_{ H^1(D_0  ) }= - \int\limits_{D_0} (A-A_0) \grad \phi_1 \cdot \grad \overline{ v_2 }-k^2 (n-n_0)\phi_1 \overline{ v_2} \, dx\\
&&\hspace*{1.5cm}= - \int\limits_{D_0} (A-A_0) \grad v_1 \cdot \grad \overline{ v_2 }-k^2 (n-n_0) v_1 \overline{ v_2} \, dx \\
&&\hspace*{1.5cm} - \int\limits_{D_0} (A-A_0) \grad u_1 \cdot \grad \overline{ v_2 }-k^2 (n-n_0) u_1 \overline{ v_2} \, dx 
\end{eqnarray*}
Now recall that for a given $v_2 \in H^1(D_0)$ we have that $u_2 \in H^1_{loc}(\R^m)$ satisfies  
\begin{eqnarray*}
&&\grad \cdot A_0 \grad u_2 +k^2 n_0 u_2=\grad \cdot (A-A_0)  \grad v_2 +k^2(n-n_0)v_2 \quad   \textrm{ in } \,  \R^m. 
\end{eqnarray*}
Hence multiplying the above equation by $\overline{u_1}$ and integrating by parts over $B_R$ such $D\subset B_R$ we have that
\begin{eqnarray}
&&\hspace*{-2cm}-\int\limits_{B_R} A_0\grad u_2 \cdot  \grad \overline{u_1} -k^2 n_0 u_2 \overline{u_1}\, dx +\int\limits_{\partial B_R} \overline{u_1}  \frac{\partial u_2}{\partial \nu}\, ds \nonumber  \\
&&\hspace*{2cm} = - \int\limits_{D_0} (A-A_0) \grad v_2 \cdot \grad \overline{ u_1 } - k^2  \int\limits_{D_0} (n-n_0) v_2 \overline{ u_1 } \, dx. 
\end{eqnarray}
By taking the conjugate of the above expression and using the Dirichlet to Neumann operator ${\mathbb T}_k$ we have that 
\begin{eqnarray*}
&&(-T v_1, v_2)_{ H^1(D_0  ) }= \int\limits_{D_0} (A-A_0) \grad v_1 \cdot \grad \overline{ v_2 }-k^2 (n-n_0)v_1 \overline{ v_2} \, dx \\
&&\hspace*{2cm}+\int\limits_{B_R} A_0 \grad u_1 \cdot  \grad \overline{u_2} -k^2 n_0 u_1 \overline{u_2}\, dx -\int\limits_{\partial B_R}  u_1\overline{{\mathbb T}_k{u_2} } \, ds\\
&&\hspace*{2cm} - \int\limits_{D_0} (A_0-\overline{A_0}) \grad u_1 \cdot \grad \overline{ v_2 }-k^2 (n_0-\overline{n_0}) u_1 \overline{ v_2} \, dx.     
\end{eqnarray*}
In order to analyze $\Re(T)$ we first compute  $(T^*v_1,v_2)_{H^1(D_0)}$ and then  $1/2(T+T^*v_1,v_2)_{H^1(D_0)}$ to obtain
\begin{eqnarray*}
&&\hspace*{-2cm}(-\Re(T) v_1, v_2)_{ H^1(D_0  ) }= \int\limits_{D_0} (A-\Re(A_0)) \grad v_1 \cdot \grad \overline{ v_2 }-k^2 (n-\Re(n_0))v_1 \overline{ v_2} \, dx \\
&&\hspace*{-1.6cm}+\int\limits_{B_R} \Re(A_0) \grad u_1 \cdot  \grad \overline{u_2} -k^2\Re( n_0) u_1 \overline{u_2}\, dx -\int\limits_{\partial B_R}  u_1\overline{\Re({\mathbb T}_k){u_2} } \, ds\\
&&\hspace*{-1.6cm}- i\int\limits_{D_0} \Im(A_0)\grad u_1 \cdot \grad \overline{ v_2 }-k^2 \Im(n_0) u_1 \overline{ v_2} \, dx   + i\int\limits_{D_0} \Im(A_0)\grad v_1 \cdot \grad \overline{ u_2 }-k^2 \Im(n_0) v_1 \overline{ u_2} \, dx.   
\end{eqnarray*}
(Note that it is easy to see that the above expression is self-adjoint despite the appearance of the complex $i$ in front of complex-valued mixed terms.)

Now define the bounded linear operators $S$ and $K: H^1(D_0) \longmapsto  H^1(D_0)$ by the Riesz representation theorem such that 
\begin{eqnarray*}
&&\hspace*{-2cm}(S v_1, v_2)_{ H^1(D_0  ) }=  \int\limits_{D_0} (A-\Re(A_0)) \grad v_1 \cdot \grad \overline{ v_2 }+v_1\overline{v_2} \,dx +\int\limits_{B_R} \Re(A_0) \grad u_1 \cdot  \grad \overline{u_2}\,dx\\ 
&&\hspace*{-1cm}-\int\limits_{\partial B_R}  u_1\overline{\Re({\mathbb T}_k){u_2} } \, ds - i\int\limits_{D_0} \Im(A_0)\grad u_1 \cdot \grad \overline{ v_2 } \, dx   + i\int\limits_{D_0} \Im(A_0)\grad v_1 \cdot \grad \overline{ u_2 } \, dx \\  
\end{eqnarray*}
and $(K v_1, v_2)_{ H^1(D_0  ) }=(-\Re(T) v_1, v_2)_{ H^1(D_0  ) }-(S v_1, v_2)_{ H^1(D_0  ) }$. Note that in the definition of $K$ there are only $L^2$-terms, hence $K$ is a compact operator due to the compact embedding of $H^1(D_0)$ into $L^2(D_0)$ and $H^1(B_R)$ into $L^2(B_R)$. Now, using that $A-\Re(A_0)>0$ and $\Re(A_0)>0$ along with the fact that the real part of the Dirichlet to Neumann  $\Re({\mathbb T}_k)$ is non-positive (see e.g. \cite{dtnref} in ${\mathbb R}^3$)  and applying Young's inequality we have 
\begin{eqnarray*}
&&\hspace*{-2cm}(S v_1, v_1)_{ H^1(D_0  ) }\geq \left((A-\Re(A_0)-\alpha |\Im(A_0)|)\grad v_1,\grad v_1\right)_{L^2(D_0)} + (v_1,v_1)_{L^2(D_0)} \nonumber\\
&&\hspace*{-1cm} +  \left((\Re(A_0)-\frac{1}{\alpha} |\Im(A_0)|)\grad v_1,\grad v_1\right)_{L^2(D_0)} \geq C\|v_1\|_{H^1(D_0)}.
\end{eqnarray*}
Provided $\alpha$ is such that $(A-\Re(A_0)-\alpha |\Im(A_0)|)>0$ uniformly in $D$ and $(\Re(A_0)-\frac{1}{\alpha} |\Im(A_0)|)\geq 0$ which prove the second part of the theorem.
\end{proof}

Now we are ready to state the main theorem of the paper which characterizes the support of defective region $D_0$ in terms of the range of the operator $\tilde{F}_{\sharp}^{1/2}$, where we define 
$$\tilde{F}_{\sharp}:= \big| \Re \big(\gamma_m^{-1} \mathcal{S}_b^* F \big) \big| + \big|\Im \big( \gamma_m^{-1} \mathcal{S}_b^* F \big) \big|:  L^2(\mathbb{S}) \longmapsto  L^2(\mathbb{S}).$$
We assume that the coefficients $A$, $A_0$, $n$ and $n_0$ satisfy the assumptions stated in Section \ref{problem}. 
\begin{theorem}\label{fmtheorem}
Assume that $k$ is not a transmission eigenvalue for $D_0$ if $\Im(A_0)=0$ otherwise the assumptions of Theorem \ref{tinjective} hold.  Furthermore assume that either $\Re(A_0)-A>0$ uniformly in $D_0$, or  $A-A_0>0$ uniformly in $D_0$ or there is some constant $\alpha>0$ such that $A-\Re(A_0)-\alpha |\Im(A_0)|>0$ uniformly in $D_0$ and $\Re(A_0)-\frac{1}{\alpha} |\Im(A_0)|\geq 0$  in $D_0$. For any $z\in \R^m$ we define $\phi_z:= \mathcal{S}^*_b \mathbb{G}^{\infty}( \cdot ,z) \in L^2(\mathbb{S})$, then 
$$ z \in D_0 \quad \text{if and only if } \quad \phi_z \in \mathcal{R}(\tilde{F}_{\sharp}^{1/2}).$$
\end{theorem}
\begin{proof}
Combining Theorems \ref{rangeh}, \ref{tinjective}, \ref{tkanalysis} and \ref{rtkanalysis}  the result follows by applying Theorem 2.15 in \cite{kirschbook} if $\Im(A_0)=0$ in $D_0$ or Theorem 2.1 in \cite{armin} if $\Im(A_0)<0$ in $D_0$  to the operator $\tilde{F}_{\sharp}$.
\end{proof}
Now let $(\lambda_i , \psi_i) \in \R^+ \times L^2(\mathbb{S})$ be an orthonormal eigensystem of $\tilde{F}_{\sharp}$ then by appealing to Picard's criterion (see e.g. Theorem 2.7 of \cite{p1}) we have the following characterization of the support of the defect $D_0$. 
\begin{corollary}\label{fmvoid}
Assume that $k$ is not a transmission eigenvalue of $D_0$ and $A, A_0,n,n_0$ satisfy the assumptions of Theorem  \ref{fmtheorem}. Then for $\phi_z:= \mathcal{S}^*_b \mathbb{G}^{\infty}( \cdot ,z)$
$$ z \in D_0 \quad \text{if and only if } \quad  \sum\limits_{i=1}^{\infty} \frac{|( \phi_z ,\psi_i)|^2}{ \lambda_i }< \infty.$$
\end{corollary}
\begin{remark}
{\em Alternatively the discussed analytical framework can be used to characterize the support of $D_0$ via the Generalized Linear Sampling Method developed in \cite{GLSM} which connects the support of $D_0$ to the solution of a minimization problem. 
}
\end{remark}
\section{Numerical Examples }
In this section we show  numerical examples  in ${\mathbb R}^2$, where a defective region is reconstructed from simulated  far-field data. To simulate the data,  we solve the direct scattering problems using a cubic finite element method with a perfectly matched layer and from this we will evaluate approximated $u^{\infty}_0$ and $u^{\infty}_b$. In the following calculations we use $N$ different incident and observation directions $d_j=\hat{x}_j=\big(\cos(\theta_j), \sin(\theta_j) \big)$ where $\theta_j$ are uniformly spaced points in $[0,2\pi)$. This leads to discretized far field operators ${\bf F_0}=\big[ u^{\infty}_0(\hat{x}_i,d_j) \big]_{i,j=1}^{N}$, ${\bf F_b}=\big[ u^{\infty}_b(\hat{x}_i,d_j) \big]_{i,j=1}^{N}$, and ${\bf F=F_0-F_b}$ where we can apply the Picard's criterion in Corollary \ref{fmvoid}. Even though the scattering operator $\mathcal{S}_b$ is unitary, due to approximation error in the discretized operator ${\bf S}_b$ we use ${\bf S}^{-1}_b$, instead of its adjoint ${\bf S}_b^{*}$  in order to minimize the error (in all our examples  we observe that $\| {\bf S}^*_b {\bf S}_b - I \| / \| {\bf S}^*_b \|^2 \approx 1.0014$).  Hence  we let ${\bf \tilde{F}}_{\sharp}=\big| \Re \big( \gamma_m^{-1} {\bf S}^{-1}_b {\bf F} \big) \big| + \big|\Im \big( \gamma_m^{-1} {\bf S}^{-1}_b {\bf F}  \big) \big|$ in the calculations along with $\phi_z=[{\bf S}^{-1}_b \mathbb{G}^{\infty}(\hat{x}_j,z)]_{j=1}^{N}$ where ${\bf S}^{-1}_b$ is computed by a LU decomposition.
The application of the factorization method requires the  computation of the far field pattern $\mathbb{G}^{\infty}(\hat{x},z)$ of the background Green's function $\mathbb{G}(\hat{x},z)$. In order to avoid dealing with singularity at the point $z$, for the case of piecewise homogeneous isotropic background in Theorem 2.1 of \cite{fmconstant} the authors provide a relation between the far field pattern of the background Green's function and the total field  due to the background media extending the mixed reciprocity relation known for homogeneous background \cite{coltonkress}. We use this relation in our examples for piecewise homogeneous background.  In the case of anisotropic media in $D$ we provide a partial result  of mixed reciprocity relation for $z\in D$ (for problems in nondestructive testing when $D$ is known, it is reasonable to consider the sampling points $z$ inside $D$). We show here the proof in ${\mathbb R}^2$. To this end let us first assume that  $A\neq I$ is constant  matrix and $n\neq 1$ is  constant in $D$.  The fundamental solution of the differential operator  $Lu:=\grad \cdot A\grad u +k^2nu$ in  $\R^2$  is given by
$$\Phi_b(x,y)=\frac{i}{4 \sqrt{\det A}} H^{(1)}_0(k \sqrt{n} |x-y|_A)$$
where $|x-y|^2_A=(x-y )^{\top}A^{-1}(x-y )$. (There is a similar definition for the fundamental solutions in $\R^3$.) 
\begin{theorem} \label{mixedrec}
Assume that $A$ is a constant positive definite matrix and $n>0$ constant. Then for $\hat{x} \in \mathbb{S}$ and $z \in D$ we have that 
$$\mathbb{G}^{\infty}(\hat{x},z)={\gamma_m} u_b(z,-\hat{x})$$
with $u_b(z,-\hat{x})$ is that solution of \eqref{bgp1}-\eqref{bgp3}.
\end{theorem}
\begin{proof}
Assume that $z \in D$ therefore we have that $\mathbb{G}(y,z)$ is a smooth radiating solution to Helmholtz equation in $\R^m \setminus \overline{D}$. So by \eqref{ffpdef} we have that 
$$\mathbb{G}^{\infty}(\hat{x},z)= \gamma_m \int\limits_{\partial D}\left( \mathbb{G}(y,z)^+ \frac{\partial }{\partial \nu_y} e^{-ik \hat{x}\cdot y} - \frac{\partial }{\partial \nu_y} \mathbb{G}(y,z)^+ \, e^{-ik \hat{x}\cdot y} \right)ds_y. $$
Now from Green's second identity we have that for $z\in D$
\begin{eqnarray}
u_b(z,-\hat{x})=\int\limits_{\partial D}  \frac{\partial }{\partial \nu_{A_y}} u^-_b(y,-\hat{x}) \, \Phi_b(y,z) - u^-_b(y,-\hat{x}) \frac{\partial }{\partial \nu_{A_y}} \Phi_b(y,z)  \, ds_y. \label{grep1}
\end{eqnarray}
Noting that the difference $\mathbb{G}(y,z) - \Phi_b(y,z)$ is a smooth solution of \eqref{bgp1} in $D$ and using again Green's second identity implies that 
$$0=\int\limits_{\partial D}  \frac{\partial }{\partial \nu_{A_y}} u^-_b(y,-\hat{x}) \, \big[\mathbb{G}(y,z)^-  - \Phi_b(y,z) \big] - u^-_b(y,-\hat{x}) \frac{\partial }{\partial \nu_{A_y}} \big[\mathbb{G}(y,z)^-  -  \Phi_b(y,z) \big]   \, ds_y. $$
By adding this identity with \eqref{grep1} gives that 
\begin{eqnarray}
\hspace*{-1cm}u_b(z,-\hat{x})&=&\int\limits_{\partial D}  \left(\frac{\partial }{\partial \nu_{A_y}} u^-_b(y,-\hat{x}) \, \mathbb{G}(y,z)^- - u^-_b(y,-\hat{x}) \frac{\partial }{\partial \nu_{A_y}} \mathbb{G}(y,z)^-\right) ds_y \nonumber  \\
&=& \int\limits_{\partial D}\left(  \frac{\partial }{\partial \nu_{y}} u^+_b(y,-\hat{x}) \, \mathbb{G}(y,z)^+ - u^+_b(y,-\hat{x}) \frac{\partial }{\partial \nu_{y}} \mathbb{G}(y,z)^+ \right)ds_y\label{grep2}
\end{eqnarray}
where the second equality is due to the continuity conditions  of the Cauchy data across  $\partial D$. Now since $u_b(z,-\hat{x})=u^s_b(z,-\hat{x}) + e^{-ik \hat{x}\cdot z}$ with $u^s_b(z,-\hat{x})$ being a radiating solution to Helmholtz equation in $\R^m \setminus \overline{D}$, once more an application of Green's second identity yields that
$$ 0= \int\limits_{\partial D}\left(  \frac{\partial }{\partial \nu_{y}} u^s_b(y,-\hat{x}) \, \mathbb{G}(y,z)^+ - u^s_b(y,-\hat{x}) \frac{\partial }{\partial \nu_{y}} \mathbb{G}(y,z)^+ \right) ds_y $$
Therefore we have that 
$$u_b(z,-\hat{x}) =  \int\limits_{\partial D}\left( \mathbb{G}(y,z)^+ \frac{\partial }{\partial \nu_y} e^{-ik \hat{x}\cdot y} - \frac{\partial }{\partial \nu_y} \mathbb{G}(y,z)^+ \, e^{-ik \hat{x}\cdot y} \right) ds_y.$$
which proves the result.
\end{proof}
\begin{remark}
{\em The proof of Theorem \ref{mixedrec} holds true for non-constant media in $\R^2$ or $\R^3$ as long as one can define the corresponding fundamental solution $\Phi_b(\cdot,\cdot)$ of the operator $Lu:=\grad \cdot A\grad u +k^2nu$ (see e.g. \cite{miranda}). }
\end{remark}
The above result gives that the $\mathbb{G}^{\infty}(\hat{x},z)$ can be approximated using the same cubic finite element method with a perfectly matched layer that is used to compute the scattered field $u^s_b$. In particular this way we compute $\mathbb{G}^{\infty}(\hat{x},z_p)$ at the sampling points $z_p$ being the mesh points in the finite element mesh. The defective region $D_0$ is visualized by plotting the indicator function $$ \mathcal{X}_{FM}(z)=\left[ \sum\limits_{i=1}^{N} \frac{| ( \phi_z ,{\bf \psi}_i)_{\ell^2}|^2}{ {\bf \lambda}_i } \right]^{-1} \qquad \qquad z\in D$$
where $({\bf \lambda}_i , {\bf \psi}_i) \in \R^+ \times \C^{N}$ is the eigensystem for the discretized operator ${\bf \tilde{F}}_{\sharp}$ defined by the discretized far field operators and scattering operator.

\medskip

\noindent {\bf Example 1.} We consider $D=[-2,2]^2$ where the defective region is a void $D_0$ (i.e $ A_0=I$ and $n_0=1$ in $D_0$) embedded in isotropic media. The coefficients in $D$ are given by $ A=0.5 I$ and $ n=3$.  We consider four examples of the void region $D_0$, namely the ball centered at the origin with radius $R=1$, the square $D_0=[-1,1]^2$, the ellipse centered at $(0.5,1)$ with axis $a=0.5$ and $b=0.3$, and two  circular voids with radius 0.3 centered at $(-1,1)$ and $(1,-1)$, respectively. Reconstructions are shown in Figure \ref{fmpix1} and Figure \ref{fmpix2}. 
In all our examples, we  use  $N=32$, i.e. $32$ incident directions and observation directions. 
\begin{figure}[H]
\centering
\includegraphics[scale=0.4]{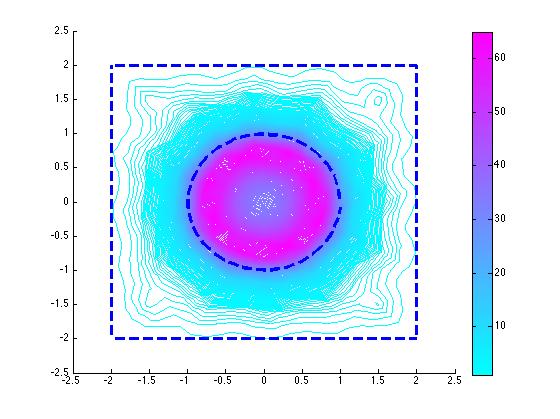}\includegraphics[scale=0.4]{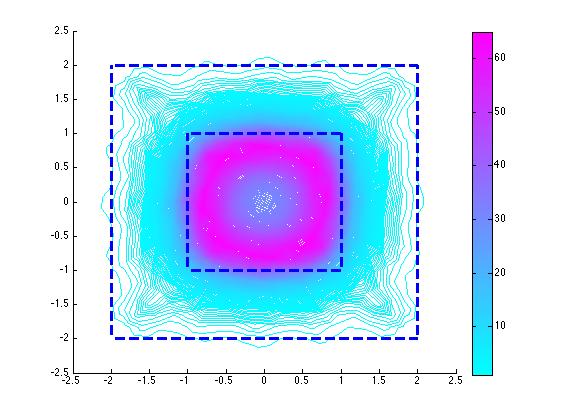}\\
\caption{On the left is the reconstruction of the circular void and on the right the square void. The defective region is a void so the coefficients are given by $A_0= I$ and  $n_0=1$ in $D_0$ for wavenumber $k=1$. Dashed line: exact boundaries of the scatterer $D$ and void(s) $D_0$. No added noise.}
\label{fmpix1}
\end{figure}
\begin{figure}[H]
\centering
\includegraphics[scale=0.4]{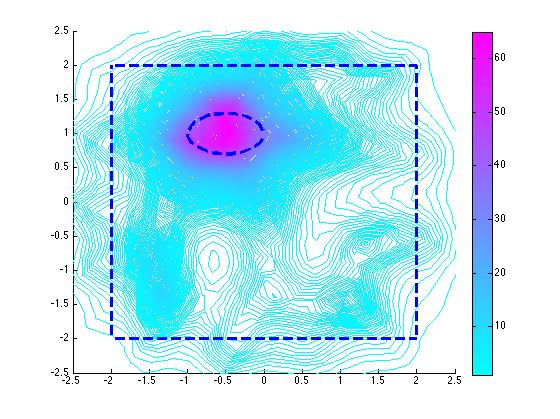}\includegraphics[scale=0.4]{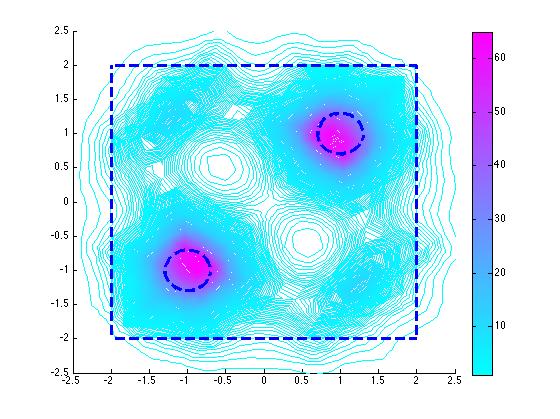}\\
\caption{ Reconstruction of the ellipse void  on the left  and of  the 2 circular voids on the right using the factorization method. The wavenumber in both examples is $k=1$. Dashed line: exact boundaries of the scatterer $D$ and void(s) $D_0$. 2\% added noise.}
\label{fmpix2}
\end{figure}

\noindent {\bf Example 2.} For this example we now reconstruct a circular void of radius 1 centered at the origin and two small circular voids in an anisotropic square scatterer $D=[-2,2]^2$. As in the previous example the  two circular voids both have radius 0.3 and they are centered at $(-1,1)$ and $(1,-1)$ respectively. The coefficients in $D$ are chosen to be given by 
$$A=\begin{pmatrix}  0.6022 & 0.1591 \\ 0.1591  &  0.7478  \end{pmatrix}$$ 
and $ n=3$ with $N=64$. The reconstructions  are presented in Figure \ref{fmpix4}.
\begin{figure}[H]
\centering
\includegraphics[scale=0.4]{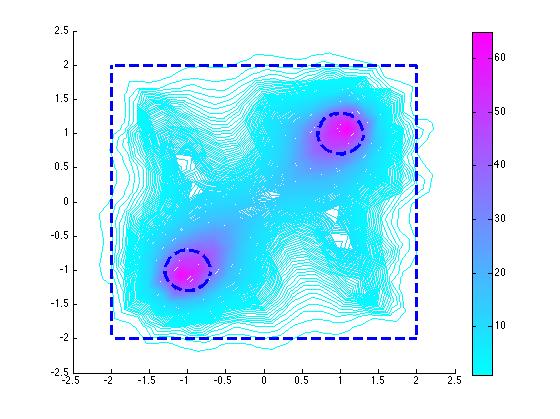}\includegraphics[scale=0.4]{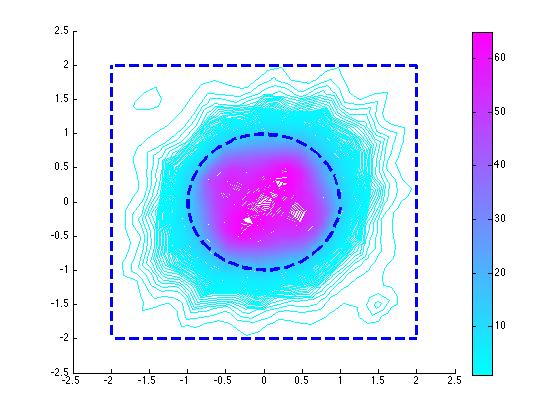}\\
\caption{ On the left is the reconstruction of the 2 circular.  While on the right is the reconstruction of the a circular void of radius 1. Where the wavenumber is $k=1$. Dashed line: exact boundaries of the scatterer and void(s). No added noise.}
\label{fmpix4}
\end{figure}
\noindent {\bf Example 3.} For our next example we now consider  anisotropic defects embedded  in anisotropic material.  In particular,  we reconstruct the two small circular defects and the ellipse inside the square $D=[-2,2]^2$. The coefficients are chosen in $D$ and $D_0$ to be given respectively by 
$$A=\begin{pmatrix}  0.6022 & 0.1591 \\  0.1591  &  0.7478 \end{pmatrix} \quad \text{ and } \quad  A_0=\begin{pmatrix}  0.1673 & -0.0308 \\  -0.0308  &  0.2030 \end{pmatrix} $$ 
and $ n= n_0=3$ with $N=64$ in both cases. The reconstructions are shown in Figure \ref{fmpix4}. 
\begin{figure}[H]
\centering
\includegraphics[scale=0.4]{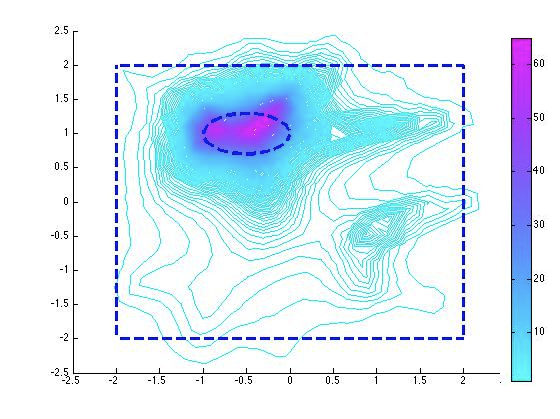}\includegraphics[scale=0.4]{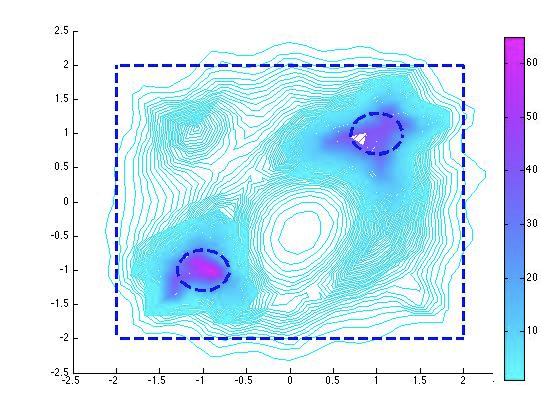}\\
\caption{  On the left is the reconstruction of the ellipse, while on the right is the reconstruction of the two discs. The wavenumber is $k=1$. Dashed line: exact boundaries of the scatterer $D$ and defect(s) $D_0$. 4\% added noise.}
\label{fmpix4}
\end{figure}

\section*{\normalsize Acknowledgments}
The research of F.~C. is supported in part by the  Air Force Office of Scientific Research  Grant FA9550-13-1-0199. The research of I.~H. is supported by the University of Delaware Graduate Fellowship. The authors would also like to thank Peter Monk for providing the code to solve the direct scattering problems and computing the far field operators.

\end{document}